\documentclass{ws-jktr}
\usepackage{subfigure}

\begin{document}
\makeatletter
\newcommand*{\rom}[1]{\expandafter\@slowromancap\romannumeral #1@}
\makeatother

\markboth{J. Li, T. J. Peters} {ISOTOPIC CONVERGENCE THEOREM}

\catchline{}{}{}{}{}

\title{ISOTOPIC CONVERGENCE THEOREM}

\author{J. Li}
\address{Department of Mathematics, University of Connecticut,\\ Storrs, CT 06269, USA.\\ hamlet.j@gmail.com}

\author{T. J. Peters}
\address{Department of Computer Science and Engineering, University of Connecticut,\\ Storrs, CT 06269, USA.\\tpeters@engr.uconn.edu}

\maketitle

\begin{abstract}
When approximating a space curve, it is natural to consider whether the knot type of the original curve is preserved in the approximant. This preservation is of strong contemporary interest in computer graphics and visualization. We establish a criterion to preserve knot type under approximation that relies upon pointwise convergence and convergence in total curvature. 
\end{abstract}

\keywords{Knot; ambient isotopy; convergence; total curvature; visualization.}

\ccode{Mathematics Subject Classification 2010: 57Q37, 57Q55, 57M25, 68R10}

\section{Introduction}
Curve approximation has a rich history, where the Weierstrass Approximation Theorem is a classical, seminal result \cite{Rudin}.  Curve approximation algorithms typically do not include any guarantees about retaining topological characteristics, such as ambient isotopic equivalence. One may easily obtain a sequence of non-trivial knots converging pointwise to a circle, with the knotted portions of the sequence becoming smaller and smaller. These non-trivial knots will never be ambient isotopic to the circle. However, ambient isotopic equivalence is a fundamental concern in knot theory. Moreover, it is a theoretical foundation for curve approximation algorithms in computer graphics and visualization.

So a natural question is what criterion will guarantee ambient isotopic equivalence for curve approximation? The answer is that, besides pointwise convergence, an additional hypothesis of convergence in total curvature will be sufficient, as we shall prove. An example is shown by Figure~\ref{fig:aa}. 

\begin{figure}[htb]
\centering
    \subfigure[Unknot vs. Knot]
    {
\includegraphics[height=4.5cm]{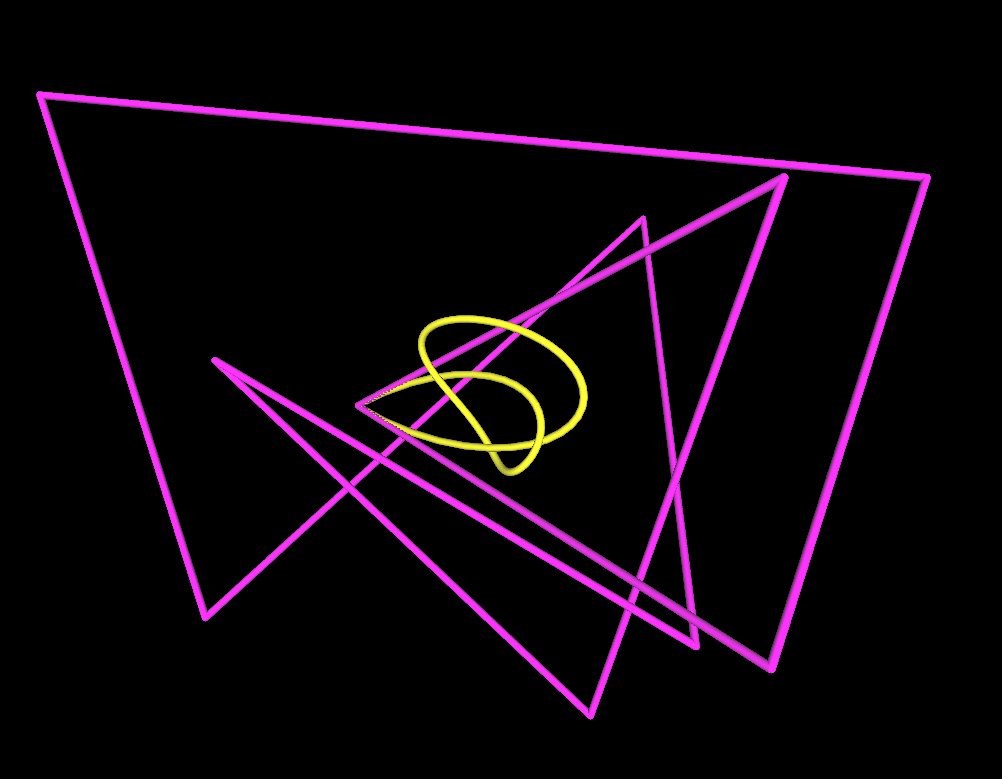}
 \label{fig:ku0}
    }
    \subfigure[Knot vs. Knot]
    {
   \includegraphics[height=4.5cm]{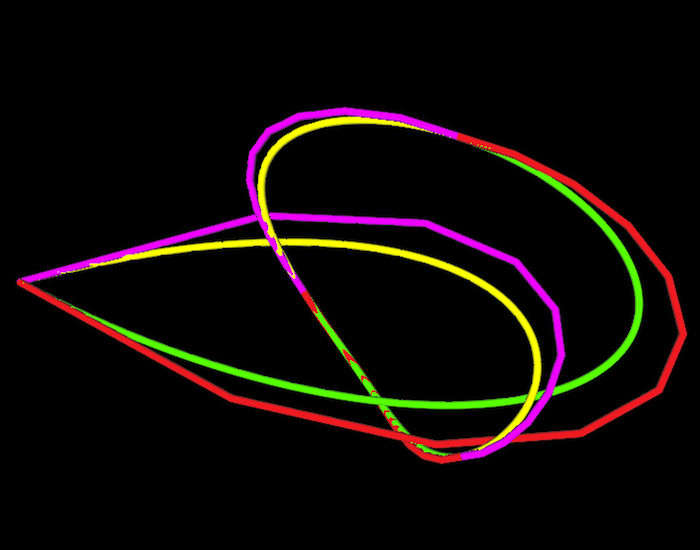}
    \label{fig:ku2}
    }
    \caption{Ambient isotopic approximation}
    \label{fig:aa}
\end{figure}

Figure~\ref{fig:ku0} shows a knotted curve (yellow) which is a trefoil, where this curve is a spline initially defined by an unknotted PL curve (purple), called a control polygon. This PL curve is often treated as the initial approximation of the spline curve. A standard algorithm, called subdivision \cite{G.Farin1990}, is used to generate new PL curves that more closely approximate the spline curve. Figure~\ref{fig:ku2} shows an ambient isotopic approximation generated by subdivision, as this PL approximation is a trefoil. 

There are three main theorems presented. All have a hypothesis of a sequence of curves converging to another smooth curve $\mathcal{C}$.  In Theorem~\ref{thm:insc}, the elements of the sequence are PL inscribed curves.  In Theorem~\ref{tim:bigcurva} and~\ref{thm:distcurv}, the class of curves is generalized to any piecewise $C^2$ curves, with the first being a technical result about a lower bound for the total curvature of elements in some tail of the sequence.  These first two results are used to provide the main result Theorem~\ref{thm:distcurv}, showing that pointwise convergence and convergence in total curvature over this richer class of piecewise $C^2$ curves produce a tail of elements that are ambient isotopic to $\mathcal{C}$.

\section{Related Work}\label{sec:rw}
The Isotopic Convergence Theorem presented here is motivated by the question about topological integrity of geometric models in computer graphics and visualization. But it is a general and pure theoretical result, dealing with the fundamental equivalence relation in knot theory, which may be applied, but extends beyond the limit of any specific applications. 

The preservation of topology in computer graphics and visualization has previously been articulated in two primary applications \cite{TJP2011}:
\begin{enumerate}
\item preservation of  isotopic equivalence by approximations; and 
\item preservation of isotopic equivalence during dynamic changes, such as protein unfolding.
\end{enumerate}

The publications \cite{Amenta2003, L.-E.Andersson2000, Lance2009, Moore_Peters_Roulier2007} are among the first that provided algorithms to ensure an ambient isotopic approximation. The paper \cite{Maekawa_Patrikalakis_Sakkalis_Yu1998} provided existence criteria for a PL approximation of a rational spline curve, but did not include any specific algorithms. 

Recent progress was made for the class of B\'ezier curves, by providing stopping criteria for subdivision algorithms to ensure ambient isotopic equivalence for B\'ezier curves of any degree $n$ \cite{JL-bez-iso}, extending the previous work of \cite{Moore_Peters_Roulier2007}, that had been restricted to degree less than $4$. This extension is based on theorems and sophisticated techniques on knot structures. 

This work here extends to a much broader class of curves, piecewise $C^2$ curves, where there is no restriction on approximation algorithms. Because of its generality, this pure mathematical result is potentially applicable to both theoretical and practical areas.

There exist results in the literature showing ambient isotopy from a different point of view \cite{DenneSullivan2008,  Sullivan2008}. Precisely, there is an upper bound on distance and an upper bound on angles between corresponding points for two curves. If the corresponding distances and angles are within the upper bounds, then they are ambient isotopic. 
 
Milnor \cite{Milnor1950} defined the total curvature for a $C^2$ curve using inscribed PL curves. The extension of the definition to piecewise $C^2$ curves can be trivially done. Consequently, Fenchel's Theorem can be applied to piecewise $C^2$ curves, as we need here. 

Milnor \cite{Milnor1950} also proved the result restricted to inscribed curves. That is a similar version of Theorem~\ref{thm:insc} presented here. That result was recently generalized to finite total curvature knots \cite{DenneSullivan2008}. The application to graphs was also established recently \cite{Gulliver2012}. Our proof here indicates upper bounds on distance and total curvature, which leads to the formulation of algorithms. 

\section{Preliminaries}
\label{sec:ddttc}
\textbf{Use $\mathcal{C}$ to denote a compact, regular, $C^2$, simple, parametric, space curve. Let $\{C_i\}_{1}^{\infty}$ denote a sequence of piecewise $C^2$, parametric curves. Suppose all curves are parametrized on $[0,1]$, that is, $\mathcal{C}=\mathcal{C}(t)$ and $C_i=C_i(t)$ for $t \in [0,1]$. Denote the sub-curve of $\mathcal{C}$ corresponding to [$a,b] \subset [0,1]$ as $\mathcal{C}_{[a,b]}$, and similarly use $C_{i[a,b]}$ for $C_i$. Denote total curvature as a function $T_{\kappa}(\cdot)$.}

\subsection{Total curvatures of piecewise $C^2$ curves}

\begin{definition}[Exterior angles of PL curves]\textup{\cite{Milnor1950}}
\label{def:exterior_angles}
The \textit{exterior angle} between two oriented line segments $\overrightarrow{P_{m-1}P_m}$ and $\overrightarrow{P_mP_{m+1}}$, is the angle between the extension of $\overrightarrow{P_{m-1}P_m}$ and $\overrightarrow{P_mP_{m+1}}$,  as shown in Figure \ref{fig:a}. Let the measure of the exterior angle to be $\alpha_m$ satisfying:
\begin{center}
$0 \leq \alpha_m \leq \pi$.
\end{center}
This definition naturally generalizes to any two vectors, $\vec{v}_1$ and $\vec{v}_2$, by joining these vectors at their initial points, while denoting the measure between them as $\eta(\vec{v}_1,\vec{v}_2)$, as indicated in Figure~\ref{fig:av}.

\end{definition}

\begin{figure}[h!]
\centering
        \subfigure[Orientated lines]
   {   \includegraphics[height=3.5cm]{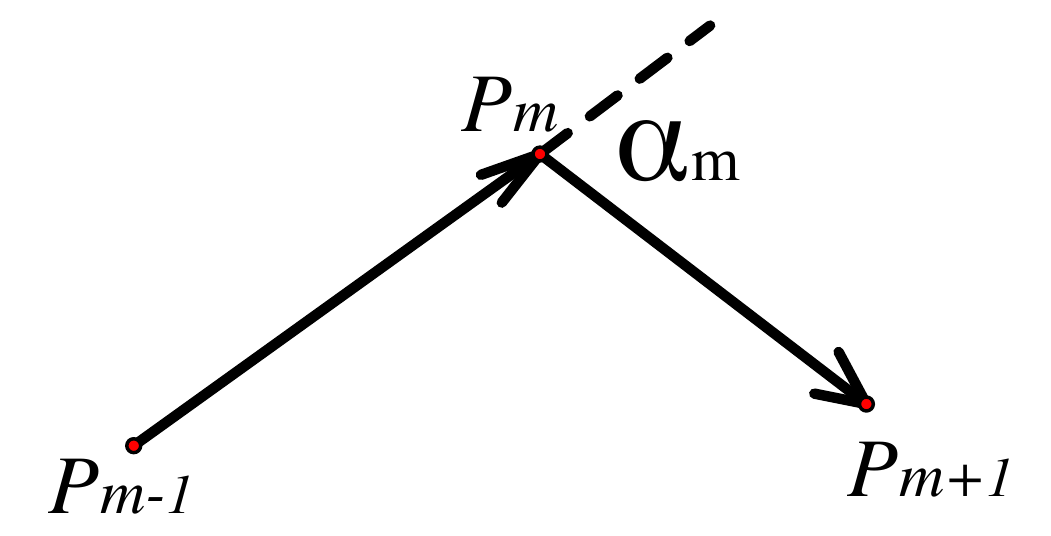}  \label{fig:a} }
        \subfigure[Vectors]
   {   \includegraphics[height=3.5cm]{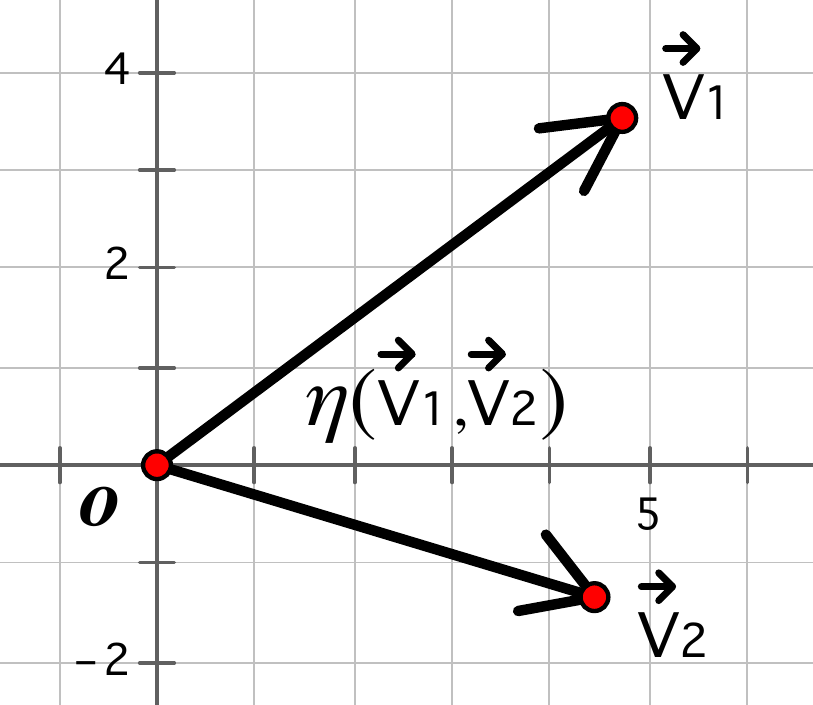}  \label{fig:av} }
    \caption{An exterior angle}    
\label{fig:alpha1}   
 \end{figure}   

The concept of exterior angle is used to unify the concept of total curvature for curves that are PL or  differentiable.  

\begin{definition}[Total curvatures of PL curves] \textup{\cite{Milnor1950}}
The total curvature of a PL curve, is  the sum of the exterior angles.  
\end{definition}

\begin{definition}[Total curvatures of $C^2$ curves] \textup{\cite{Milnor1950}}
Parametrize a $C^2$ curve with arc length $s$ on $[0,\ell]$ and use $\kappa(s)$ to denote the curvature. Then the total curvature of the curve is $\int_0^{\ell} | \kappa(s)| \ ds.$
\end{definition}

\begin{definition}[Exterior angles of piecewise $C^2$ curves]
\label{def:exall}
For a piecewise $C^2$ curve $\gamma(t)$, define the exterior angle at some $t_i$ to be the exterior angle formed by $\gamma'(t_i-)$ and  $\gamma'(t_i+)$ where
$$\gamma'(t_i-) = \lim_{h \rightarrow 0} \frac{\gamma(t_i)-\gamma(t_i-h)}{h},$$
and $$\gamma'(t_i+) = \lim_{h \rightarrow 0} \frac{\gamma(t_i+h)-\gamma(t_i)}{h}.$$
\end{definition}

\begin{definition}\footnote{This is similarly defined in a recent paper \cite{Gulliver2012}.} [Total curvatures of piecewise $C^2$ curves]
\label{def:ttc}
Suppose that a piecewise $C^2$ curve $\phi(t)$ (regular at the $C^2$ points) is not $C^2$ at finitely many parameters $t_1,\cdots,t_n$. Denote the sum of the total curvatures of all the $C^2$ sub-curves as $T_{\kappa1}$, and the sum of exterior angles at $t_1,\cdots,t_n$ as $T_{\kappa2}$. Then the total curvature of $\phi(t)$ is $T_{\kappa1} + T_{\kappa2}$.
\end{definition}

\subsection{Definitions of convergence}

\begin{definition}\label{def:indist}
We say that  $\{C_i\}_{1}^{\infty}$ converges to $\mathcal{C}$ in {\em parametric measure distance} if for any $\epsilon>0$, there exists an integer $N$ such that $\max_{t \in [0,1]} |C_i(t)-\mathcal{C}(t)| < \epsilon$ for all $i \geq N$. 
\end{definition}

\begin{remark}
For compact curves, this convergence in parametric measure distance is equivalent to pointwise convergence.
\end{remark}

\begin{definition} \textup{\cite{J.Munkres1999} }
Let $X$ and $Y$ be two non-empty subsets of a metric space $(M, d)$. We define their Hausdorff distance $\mu(X,Y)$ by
 $$\max\{\sup_{x\in X} \inf_{y\in Y} d(x,y), \sup_{y\in Y} \inf_{x\in X} d(x,y) \}.$$
\end{definition}

\begin{remark}
By the definition of Hausdorff distance, the pointwise convergence implies the convergence in Hausdorff distance. 
\end{remark}

\begin{definition}\label{def:tc}
We say that  $\{C_i\}_{1}^{\infty}$ converges to $\mathcal{C}$ in total curvature if for any $\epsilon>0$, there exists an integer $N$ such that $|T_{\kappa}(C_i)-T_{\kappa}(\mathcal{C})| < \epsilon$ for all $i \geq N$. We designate this property as \textit{convergence in total curvature}.
\end{definition}

\begin{definition}\label{def:ltc}
We say that  $\{C_i\}_{1}^{\infty}$ uniformly converges to $\mathcal{C}$ in total curvature if for any $[t_1,t_2] \subset [0,1]$ and $\forall \epsilon>0$, there exists an integer $N$ such that whenever $i \geq N$, $|T_{\kappa}(C_{i[t_1,t_2]})-T_{\kappa}(\mathcal{C}_{[t_1,t_2]})| < \epsilon$. We designate this property as  \textit{uniform convergence in total curvature}.
\end{definition}

\begin{remark}
Uniform convergence in total curvature implies convergence in total curvature. But the converse is not true. 
\end{remark}


\section{Isotopic Convergence of Inscribed PL Curves}
\label{sec:inc}

We will use the concept of PL inscribed curves as previously defined \cite{Milnor1950}.
\begin{definition} \label{def:ins} 
A closed PL curve $L$ with vertices $v_1,v_2,\cdots,v_m$ is said to be inscribed in  curve $\mathcal{C}(t)$ if there is a sequence $\{t_j\}_{1}^m$ of parameter values such that $v_i=\mathcal{C}(t_j)$ for $j=1,2,\cdots,m$.
We parametrize $L$ over $[0,1]$, denoted as $L(t)$, by
$$
L(t_j)=v_j\ for\ j=0,1,\cdots,m
$$
and $L(t)$ interpolates linearly between vertices.
\end{definition}

The previously established results \cite[Theorem 2.2]{Milnor1950} and \cite[Proposition 3.1]{Sullivan2008} showed that a sequence of finer and finer inscribed PL curves will converge in total curvature. The uniform convergence in total curvature follows easily. For the sake of completeness, we present the proof here. 

\begin{lemma} \label{lem:app}
For a piecewise $C^2$ curve $\gamma(t)$ parametrized on $[0,1]$ (which is regular at all $C^2$ points), a sequence $\{L_i\}_1^{\infty}$ of inscribed PL curves can be chosen such that $\{L_i\}_1^{\infty}$ pointwise converges to $\gamma$ and uniformly converges to $\gamma$ in total curvature.
\end{lemma}

\begin{proof}
We first take the end points $\gamma(t_0) = \gamma(0)$ and $\gamma(t_n) = \gamma(1)$. And then select\footnote{Acute readers may find later that this choice of points is sufficient for this lemma, but not necessary. This choice is for ease of exposition.} the points where $\gamma$ fails to be $C^2$. Denoted these points as $\{\gamma(t_0), \gamma(t_1),\cdots,\gamma(t_{n-1}), \gamma(t_n)\}$. We then compute midpoints: $\gamma(\frac{t_j+t_{j+1}}{2})$ for $j\in \{0,1,\ldots,n-1\}$ to form $L_2$ which is determined by vertices: 
$$\{\gamma(t_0), \gamma(\frac{t_0+t_1}{2}), \gamma(t_1), \cdots, \gamma(t_{n-1}), \gamma(\frac{t_{n-1}+t_n}{2}), \gamma(t_n)\}.$$ 
Continuing this process, we obtain a sequence $\{L_i\}_1^{\infty}$ of inscribed PL curves. 

Suppose the set of vertices of $L_i$ is $\{ v_{i,k}=\gamma(t_{i,k}) \}$, for some finitely many parameter values $t_{i,k}$. Use uniform parametrization \cite{Morin_Goldman2001} for $L_i$ such  that $v_{i,k}=L_i(t_{i,k})$, and points between each pair of consecutive vertices are interpolated linearly. Note first that this process implies that $\{L_i\}_1^{\infty}$ pointwise converges to $\mathcal{C}$. For the uniform convergence in total curvature, consider the following:

\begin{enumerate}
\item Consider each $t_j$ where $\gamma$ fails to be $C^2$. Denote the parameters of two vertices of $L_i$ adjacent to $L_i(t_j)$ as $t^i_{j1}$ and  $t^i_{j2}$. Note that $\lim _{i \rightarrow \infty}t^i_{j1}=\lim _{i \rightarrow \infty}t^i_{j2}=t_j$. This implies that the slope of $\overrightarrow{L_i(t^i_{j1})L_i(t_j)}$ and the slope of $\overrightarrow{L_i(t^i_{j2})L_i(t_j)}$ go to $\gamma'(t_j-)$ and $\gamma'(t_j+)$ respectively. This shows that
$$\lim_{i \rightarrow \infty} \eta(\overrightarrow{L_i(t^i_{j1})L_i(t_j)}, \overrightarrow{L_i(t^i_{j2})L_i(t_j)})= \eta(\gamma'(t_j-),\gamma'(t_j+)).$$

\item For a $C^2$ sub-curve of $\gamma$, the proof of \cite[Theorem 2.2] {Milnor1950} shows that the total curvatures of the corresponding inscribed PL curves converge to the total curvature of the $C^2$ sub-curve. 
\end{enumerate}

By Definition~\ref{def:ttc}, the above (1) and (2) together imply the uniform convergence in total curvature.  
\end{proof}

Since uniform convergence in total curvature implies  convergence in total curvature (Definition~\ref{def:ltc}), the corollary below follows immediately.
\begin{corollary}\label{coro:app} \textup{\cite[Theorem 2.2]{Milnor1950} \cite[Proposition 3.1]{Sullivan2008}}
For $\mathcal{C}$, a sequence $\{L_i\}_1^{\infty}$ of inscribed PL curves can be chosen such that $\{L_i\}_1^{\infty}$ converges to $\mathcal{C}$ pointwise and in total curvature.
\end{corollary}

\begin{theorem} [Fenchel's Theorem] \textup{\cite{Milnor1950}} \label{thm:gen-fenchel}
The total curvature of a closed curve is at least $2 \pi$, with equality holding if and only if the curve is convex.
\end{theorem}



\begin{lemma}\label{lem:npwint}
Denote the plane normal to $\mathcal{C}$ at some $t_0 \in (0,1)$ as $\Pi(t_0)$. Consider two sub-curves $\mathcal{C}_{[t_0-u]}$ and $\mathcal{C}_{[t_0+v]}$ for some $u\in (0,t_0)$ and $v\in (t_0,1)$. If both $T_{\kappa}(\mathcal{C}_{[t_0-u]})<\frac{\pi}{2}$ and $T_{\kappa}(\mathcal{C}_{[t_0+v]})<\frac{\pi}{2}$, then these two sub-curves $\mathcal{C}_{[t_0-u]}$ and $\mathcal{C}_{[t_0+v]}$ are separated by $\Pi(t_0)$ except at $\mathcal{C}(t_0)$.
\end{lemma}

\begin{proof}
Denote the point $\mathcal{C}(t_0)$ as $a$. Suppose that the conclusion is false, then either $\mathcal{C}_{[t_0-u]}$ or $\mathcal{C}_{[t_0+v]}$ intersects $\Pi(t_0)$ other than at $a$. Assume without loss of generality that $\mathcal{C}_{[t_0+v]} \cap \Pi(t_0)$ contains another point, denoted as $b$. Then the sub-curve $\mathcal{C}_{[t_0+v]}$ and the line segment $\overline{ab}$ form a closed curve $\mathcal{C}_{[t_0+v]} \cup \overline{ab}$. So $T_{\kappa}(\mathcal{C}_{[t_0+v]} \cup \overline{ab}) \geq 2\pi$ by Theorem~\ref{thm:gen-fenchel}. 

Denote the exterior angles at $a$ and $b$ as $\alpha$ and $\beta$ respectively. Then $\alpha=\frac{\pi}{2}$ since $\Pi(t_0)$ is normal to $\mathcal{C}'(t_0)$. By Definition~\ref{def:exterior_angles}, $\beta \leq \pi$. By Definition~\ref{def:ttc} we have
$$T_{\kappa}(\mathcal{C}_{[t_0+v]} \cup \overline{ab}) =T_{\kappa}(\mathcal{C}_{[t_0+v]}) + \alpha + \beta \leq T_{\kappa}(\mathcal{C}_{[t_0+v]}) + \frac{\pi}{2} +\pi.$$
So $$T_{\kappa}(\mathcal{C}_{[t_0+v]}) + \frac{\pi}{2} +\pi \geq 2\pi.$$
Therefore 
$$T_{\kappa}(\mathcal{C}_{[t_0+v]}) \geq \frac{\pi}{2},$$
which is a contradiction.  
\end{proof}

Theorem~\ref{thm:insc} below is restricted to ``inscribed PL curves". The general theorem of ``piecewise $C^2$ curves, either inscribed or not" will be established later in Theorem~\ref{thm:distcurv}. 

\begin{theorem}\label{thm:insc}
For any sequence $\{L_i\}_1^{\infty}$ of inscribed PL curves that pointwise converges to $\mathcal{C}$ and uniformly converges to $\mathcal{C}$ in total curvature, a positive integer $N$ can be found as below such that for all $i > N$, $L_i$ is ambient isotopic to $\mathcal{C}$.
\end{theorem}
\begin{proof}
For $\mathcal{C}$, there is a non-self-intersecting tubular surface\footnote{We use the terminology of \textit{tubular surface} as generalization from the recent usage \cite{Maekawa_Patrikalakis_Sakkalis_Yu1998} regarding the classically defined \textit{pipe surface} \cite{Monge}. } of radius $r$ \cite{Maekawa_Patrikalakis_Sakkalis_Yu1998}. 

Pointwise convergence and the uniform convergence in total curvature imply that there exists a positive integer $N$ such that for an arbitrary $i>N$: 
\begin{enumerate}
\item The PL curve $L_i$ lies inside of the tubular surface of radius $r$; and 
\item Denote the set of vertices of $L_i$ as $\{v_j\}_{j=0}^n$. Suppose the sub-curve of $\mathcal{C}$ between two arbitrary consecutive vertices $v_j$ and $v_{j+1}$ as $\mathcal{A}_j$, for $j=0,\ldots,n-1$. Then since the total curvature of $\overrightarrow{v_jv_{j+1}}$ is $0$, the total curvature of $\mathcal{A}_j$ can be less than $\frac{\pi}{2}$.
\end{enumerate}

Lemma~\ref{lem:npwint} implies that all such sub-curves $\mathcal{A}_j$ are separated by normal planes except the connection points. The facts about fitting inside a tubular surface and separation by normal planes provide a sufficient condition \cite{Maekawa_Patrikalakis_Sakkalis_Yu1998} for $L_i$ being ambient isotopic to $\mathcal{C}$.
\end{proof}

\begin{remark}\label{rmk:pipe}
The paper \cite{Maekawa_Patrikalakis_Sakkalis_Yu1998} provides the computation of the radius $r$ only for rational spline curves. However, the method of computing $r$ is similar for other compact, regular, $C^2$, and simple curves, that is, setting
$$r < \min \{\frac{1}{\kappa_{max}}, \frac{d_{min}}{2}, r_{end}\},$$
where $\kappa_{max}$ is the maximum of the curvatures, $d_{min}$ is the minimum separation distance, and $r_{end}$ is the maximal radius around the end points that does not yield self-intersections.
\end{remark}


\section{Pointwise Convergence}
\label{sec:cdo}
Pointwise convergence provides a lower bound of the total curvatures of approximants (Theorem~\ref{tim:bigcurva}). The proof relies upon showing this for PL curves first (Lemma~\ref{lem:egre}). The technique used here is the well known ``2D push" \cite{Bing1983}. It is sufficient here to consider a specialized type of push, designated, below, as a {\em median push}. 

\begin{definition}\label{def:mp}
Assume that triangle $\triangle{ABC}$ has non-collinear vertices $A,B$ and $C$. Push a vertex, say $B$, along the corresponding median of the triangle to the midpoint of the side $AC$. We call this specific kind of ``2D push", a median push. 
\end{definition}

\begin{lemma}\label{lem:egre}
Let $\{L_i\}_{i=1}^{\infty}$ be a sequence of PL curves parametrized on $[0,1]$ and $L$ be a PL curve parametrized on $[0,1]$. If $\{L_i\}_{i=1}^{\infty}$ pointwise converges to $L$, then for  $\forall \epsilon>0$, there exists an integer $N$ such that $T_{\kappa}(L_i)> T_{\kappa}(L)-\epsilon$ for all $i \geq N$.
\end{lemma}

\begin{proof}
For an arbitrary vertex $v$ of $L$, suppose $v=L(t_v)$ for some $t_v \in [0,1]$. Let $B_v$ be a closed ball centered at $v$. Since $L$ is a compact PL curve, we can choose the radius of $B_v$ small enough such that: 

\begin{enumerate}
\item the ball $B_v$ contains only the single vertex $v$ of $L$; and  
\item it intersects only the two line segments of $L$ which are connected at $v$. Denote these intersections as $u=L(t_u)$ and $w=L(t_w)$ for some $t_u, t_w \in [0,1]$. Then $u, w$ and $v$ together form a triangle $\triangle{uvw}$.
\end{enumerate}

Let $u_i=L_i(t_u)$, $v_i=L_i(t_v)$ and $w_i=L_i(t_w)$. Denote the exterior angle of the triangle $\triangle{uvw}$ at $v$ as $\eta(v)$, and correspondingly the exterior angle of $\triangle{u_iv_iw_i}$ at $v_i$ as $\eta(v_i)$. (Note that $\eta(v)$ is not necessarily equal to the exterior angle of $L$ at $v$. Similarly for $\eta(v_i)$.) By the pointwise convergence we have that the triangle $\triangle{u_iv_iw_i}$ converges to $\triangle{uvw}$. So $\eta(v_i)$ converges to $\eta(v)$. That is, for $\forall \epsilon'>0$ there exists an $N$ such that $\eta(v_i)> \eta(v)-\epsilon'$ for all $i \geq N$.

Consider the PL sub-curve of $L_i$ lying in $B_v$ and denote its total curvature as $T_{\kappa}(L_i \cap B_v)$. This PL sub-curve of $L_i$ can be reduced by median pushes to $\triangle{u_iv_iw_i}$. The existing result \cite[Lemma 1.1, Corollary 1.2]{Milnor1950} implies that $T_{\kappa}(L_i \cap B_v) \geq \eta(v_i)$. So for $i \geq N$,
\begin{align}\label{eq:tklcu}T_{\kappa}(L_i \cap B_v) > \eta(v)-\epsilon'.\end{align}

Denote the set of vertices of $L$ as $V$. Then $T_{\kappa}(L)=\sum_{v \in V} \eta(v)$. Note that $T_{\kappa}(L_i) \geq \sum_{v \in V} T_{\kappa}(L_i \cap B_v)$. So Inequality~\ref{eq:tklcu} implies that 
$$T_{\kappa}(L_i) \geq \sum_{v \in V} T_{\kappa}(L_i \cap B_v) > \sum_{v \in V} \eta(v)-\epsilon'n = T_{\kappa}(L)-\epsilon' n$$ 
where $n$ is the number of vertices of $L$. Let $\epsilon'=\frac{\epsilon}{n}$, then we complete the prove. 
\end{proof}

\begin{theorem} \label{tim:bigcurva}
If  $\{C_i\}_{1}^{\infty}$ pointwise converges to $\mathcal{C}$, then for $\forall \epsilon>0$, there exists an integer $N$ such that $T_{\kappa}(C_i)> T_{\kappa}(\mathcal{C})-\epsilon$ for all $i \geq N$.
\end{theorem}

\begin{proof}
By Lemma~\ref{lem:app}, we can use inscribed PL curves to approximate $\{C_i\}_{1}^{\infty}$ and $\mathcal{C}$, such that the approximations converge pointwise and in total curvature. Then apply the Lemma~\ref{lem:egre} to these inscribed PL curves. Since these inscribed PL curves converge pointwise and in total curvature to $\{C_i\}_{1}^{\infty}$ and $\mathcal{C}$ respectively, the desired conclusion follows.
\end{proof}


\section{Uniform Convergence in Total Curvature}
Convergence in total curvature is weaker than uniform convergence in total curvature. But pointwise convergence and convergence in total curvature together imply the uniform convergence, which is shown by Lemma~\ref{lem:loc-cov} below. 

\begin{lemma}\label{lem:loc-cov}
If $\{C_i\}_{1}^{\infty}$ converges to a $C^2$ curve $\mathcal{C}$ pointwise and in total curvature, then $\{C_i\}_{1}^{\infty}$ uniformly converges to $\mathcal{C}$ in total curvature.
\end{lemma}

\begin{proof}
Assume not, then there exist a subset $[t_1,t_2] \subset [0,1]$ and a $\tau>0$ such that for any integer $N$, there is a $i \geq N$ such that  $|T_{\kappa}(C_{i[t_1,t_2]})-T_{\kappa}(\mathcal{C}_{[t_1,t_2]})| > \tau$, that is $T_{\kappa}(C_{i[t_1,t_2]}) >T_{\kappa}(\mathcal{C}_{[t_1,t_2]})+ \tau$ or $T_{\kappa}(C_{i[t_1,t_2]}) <T_{\kappa}(\mathcal{C}_{[t_1,t_2]})- \tau$. The latter is precluded by Theorem~\ref{tim:bigcurva}. Therefore
\begin{align}\label{eq:kt} T_{\kappa}(C_{i[t_1,t_2]}) >T_{\kappa}(\mathcal{C}_{[t_1,t_2]})+ \tau. \end{align}

Consider the sequence of the sub-curves of $\{C_i\}_{1}^{\infty}$ restricted to the complement $[t_1,t_2]^c$ of $[t_1,t_2]$, and denote it as $\{C_{i[t_1,t_2]^c}\}_{1}^{\infty}$. By theorem~\ref{tim:bigcurva}, for $\frac{\tau}{2}$, there exists an integer, say $M$ such that for all $i \geq M$,
\begin{align}\label{eq:ktc}T_{\kappa}(C_{i[t_1,t_2]^c}) > T_{\kappa}(\mathcal{C}_{[t_1,t_2]^c})-\frac{\tau}{2}.\end{align}
Note that $T_{\kappa}(C_i) \geq T_{\kappa}(C_{i[t_1,t_2]}) + T_{\kappa}(C_{i[t_1,t_2]^c})$. So Equations~\ref{eq:kt} and~\ref{eq:ktc} imply that there is a $i \geq M$ so that
$$T_{\kappa}(C_i) \geq T_{\kappa}(C_{i[t_1,t_2]}) + T_{\kappa}(C_{i[t_1,t_2]^c}) > T_{\kappa}(\mathcal{C}_{[t_1,t_2]}) +  T_{\kappa}(\mathcal{C}_{[t_1,t_2]^c}) +\frac{\tau}{2}.$$ 
Since $\mathcal{C}$ is $C^2$, $T_{\kappa}(\mathcal{C}_{[t_1,t_2]}) +  T_{\kappa}(\mathcal{C}_{[t_1,t_2]^c})=T_{\kappa}(\mathcal{C})$. Therefore we get
$$T_{\kappa}(C_i) \geq T_{\kappa}(\mathcal{C})+\frac{\tau}{2},$$
which contradicts the convergence in total curvature. 
\end{proof}

\section{Isotopic Convergence}
\label{sec:cdtc}
For a $C^2$ compact curve $\mathcal{C}$, we shall, without loss of generality (Theorem~\ref{thm:insc}), consider a sequence $\{L_i\}_1^{\infty}$ of PL curves (instead of piecewise $C^2$ curves) as its approximation. We shall divide $\mathcal{C}$ into finitely many sub-curves, and reduce the corresponding sub-curves of $L_i$ to line segments, by median pushes, so as to preserve isotopic equivalence. The line segments generated by the pushes form a polyline. We shall then prove the polyline is ambient isotopic to $\mathcal{C}$. 

To get to the major theorem, we need to first establish some preliminary topological results. We use $CH(\cdot)$ to denote the convex hull of a set. 

\begin{lemma}\label{lem:tac}
Let $X$ and $Y$ be compact subspaces of an Euclidean space $\mathbb{R}^d$. If $X \cap Y=\emptyset$, then $Y$ can be subdivided into finitely many subsets, denoted as $Y_1, \ldots Y_i, \ldots Y_m$ for some $m > 0$, such that $CH(Y_i) \cap X =\emptyset$ for each $i$. 
\end{lemma}

\begin{proof}
Since $X$ is compact, for $\forall y \in Y$, $\inf_{x\in X} ||x-y|| > 0$, and hence $\exists$  an open ball $B_y \subset \mathbb{R}^d$ of $y$ such that $B_y \cap X=\emptyset$. Since $Y$ is compact, among these open balls, there are finitely many, denoted by $B_{y_1},\cdots,B_{y_m}$ such that $Y \subset \bigcup_{i=1}^m B_{y_i}$. 

Let $Y_i=Y \cap B_{y_i}$ for each $i = 1, \ldots, m$ so that 
$$CH(Y_i)=CH(Y \cap B_{y_i}) \subset CH(B_{y_i})=B_{y_i}.$$
Thus, for each $i$, we have $CH(Y_i) \cap X=\emptyset$.
\end{proof}

As we mentioned before, for a simple $C^2$ curve $\mathcal{C}$, there is a non-self-intersecting tubular surface of radius $r$ (Remark~\ref{rmk:pipe}). This surface determines a tubular neighborhood of $\mathcal{C}$, denoted as $\Gamma_{\mathcal{C}}$. Denote a sub-curve of $\mathcal{C}$ as $\mathcal{C}^k$, and the corresponding tubular neighborhood of $\mathcal{C}^k$ as $\Gamma^k$. 

\begin{lemma}
The compact curve $\mathcal{C}$ can be divided into finitely many sub-curves, denoted as $\mathcal{C}^1, \ldots, \mathcal{C}^k, \ldots, \mathcal{C}^n$ for some $n>0$, such that 
\begin{itemize}
\item $T_{\kappa}(\mathcal{C}^k) < \frac{\pi}{2}$; and
\item $CH(\mathcal{C}^k) \subset \Gamma^k$.
\end{itemize}
\end{lemma}

\begin{proof}
By Lemma~\ref{lem:tac}, $\mathcal{C}$ can be partitioned into finitely many non-empty sub-curves, each which is disjoint from $S_r(\mathcal{C})$.  Since $\mathcal{C}$, is also of finite total curvature, we can denote these sub-curves as $\mathcal{C}^1, \ldots, \mathcal{C}^k, \ldots, \mathcal{C}^n$ for some $n>0$, such that for each $k=1,\cdots,n$, $T_{\kappa}(\mathcal{C}^k) < \frac{\pi}{2}$ and $CH(\mathcal{C}^k) \cap S_r(\mathcal{C}) = \emptyset$. 

Consider $\mathcal{C}^k$ for an arbitrary $k = 1, \dots, n$ and denote the distinct normal planes at the endpoints of $\mathcal{C}^k$ by $\Pi_1, \Pi_2$, respectively.  Denote the closed convex subspace of $\mathbb{R}^3$ that contains $\mathcal{C}^k$ and is bounded by $\Pi_1$ and $\Pi_2$ as $H^k$. It is clear that $CH(\mathcal{C}^k) \subset H^k$, but since $CH(\mathcal{C}^k) \cap S_r(\mathcal{C}) = \emptyset$, we have that $CH(\mathcal{C}^k) \subset \Gamma^k$.  
\end{proof}

For $k = 1, \ldots, n$, let $[t_{k-1}, t_k]$ be the subinterval whose image is $\mathcal{C}^k$, with corresponding $\Gamma^k$. Let $\epsilon$ be real valued such that 
$$0< \epsilon < \min_{k\in \{0,\ldots,n\}} \frac{|t_k-t_{k-1}|}{2}.$$ 
We extend\footnote{If $\mathcal{C}$ is open and $t_{k-1}=0$ or $t_k=1$, consider $[0, t_k+\epsilon]$ or $[t_{k-1}-\epsilon, 1]$.} $[t_{k-1}, t_k]$ to $[t_{k-1}-\epsilon, t_k+\epsilon]$, and denote the tubular neighborhood corresponding to the extended subinterval as $\Gamma^k_{\epsilon}$, then $\Gamma^k_{\epsilon}$ only intersects $\Gamma^{k+1}_{\epsilon}$ and $\Gamma^{k-1}_{\epsilon}$ for each $k$. 

For a sequence of PL curves $\{L_i\}_{1}^{\infty}$ converging to $\mathcal{C}$ pointwise and in total curvature, denote the sub-curve of $L_i$ corresponding (with the same parameters) to $\mathcal{C}^k$ as $L_i^k$. Denote the end points of $L_i^k$ by $u^{k-1}_i$ and $u^k_i$, with the corresponding end points of $\mathcal{C}^k$ by $v^{k-1}$ and $v^k$. 

\begin{lemma}\label{lem:cds}
A large positive integer $N$ can be found such that whenever $i \geq N$, for each $k$, we have 
\begin{enumerate}
\item $T_{\kappa}(L_i^k) < \frac{\pi}{2}$;
\item $CH(L_i^k) \subset \Gamma^k_{\epsilon}$; and
\item $|u^k_i-v^k|<\frac{r}{2}$ and $\mu(\overline{u^{k-1}_iu^k_i}, \mathcal{C}^k)<\frac{r}{2}$, where $\mu(\cdot)$ refers to the Hausdorff distance. 
\end{enumerate}
\end{lemma}
\begin{proof}
The first condition follows from the uniform convergence in total curvature (Lemma~\ref{lem:loc-cov}), and the second and third follow from pointwise convergence. 
\end{proof}

Now we are ready to reduce each $L_i^k$ to the segment $\overline{u_i^{k-1}u_i^k}$ by median pushes. In order to prove there is no self-intersection of $L_i$ during the pushes, we present two lemmas below. The following lemma was established by a recent preprint \cite{JL-ang-conv}. For the sake of completeness, we give the sketch of the proof here. 
\begin{lemma}[Non-self-intersection criteria] \textup{\cite{JL-ang-conv}} \label{lem:non-int}
Let $P=(P_0,P_1,\cdots,P_n)$ be an open PL curve in $\mathbb{R}^3$.  If $ T_{\kappa}(P)<\pi$, then $P$ is simple.
\end{lemma}
\begin{proof}
Assume to the contrary that $P$ is self-intersecting. Then there must exist at least one closed loop. Consider a closed loop. By Fenchel's theorem, the total curvature of the closed loop is at least $2\pi$. The total curvature is the sum of the exterior angles, among which at most one angle is not counted as an exterior angle of $P$. But an exterior angle is less than $\pi$. So the total curvature of $P$ is at lest $2\pi-\pi=\pi$, which is a contradiction. 
\end{proof}

Milnor \cite{Milnor1950} showed the total curvature remains the same or decreases ``after" deforming a triangle to a line segment, and this can be trivially extended to show that the total curvature remains the same or decreases ``during" the whole process of deforming a triangle to a line segment, as expressed in Lemma~\ref{lem:dec-angles}.

\begin{lemma} \label{lem:dec-angles}
If a vertex of a PL curve in $\mathbb{R}^3$ undergoes a median push, then the total curvatures of new open PL curves formed during the push remain the same or decrease\footnote{This holds not only for the median push, but also for any push with a trace lying on the interior of a triangle indicated in Definition~\ref{def:mp}. }. 
\end{lemma}

\begin{lemma}\label{lem:dtel}
For each $k = 1, \ldots, n$, use median pushes to reduce $L_i^k$ to the line segment $\overline{u_i^{k-1}u_i^k}$. Then during these pushes, $L_i$ remains simple, and hence the resultant PL curve $\bigcup_{k=1}^n \overline{u_i^{k-1}u_i^k}$ is ambient isotopic to the original PL curve $L_i$.
\end{lemma}

\begin{proof}
Note that the condition (1) in Lemma~\ref{lem:cds} implies that $T_{\kappa}(L_i^{k-1} \cup L_i^k) < \pi$ and $T_{\kappa}(L_i^k \cup L_i^{k+1}) < \pi$. Lemma~\ref{lem:non-int} and~\ref{lem:dec-angles} show that the pushed $L_i^k$ does not intersect its neighbors $L_i^{k+1}$ or $L_i^{k-1}$. Since $CH(L_i^k) \subset \Gamma^k_{\epsilon}$ (Condition (2)), and $\Gamma^k_{\epsilon}$ does not intersect $\Gamma^j_{\epsilon}$ for $j\neq k-1$ or $k+1$,  the perturbed $L_i^k$ stays inside $\Gamma^k_{\epsilon}$ and does not intersect $L_i^j$ for $j\neq k-1$ or $k+1$. Then the conclusion follows. 
\end{proof}

For each $k=1,\ldots,n$, connecting the end points $v^{k-1}$ and $v^k$ of $\mathcal{C}^k$, we obtain the polyline $\bigcup_{k=1}^n \overline{v^{k-1}v^k}$.

\begin{lemma}\label{lem:tpbk}
The polyline $\bigcup_{k=1}^n \overline{v^{k-1}v^k}$ is ambient isotopic to $\bigcup_{k=1}^n \overline{u_i^{k-1}u_i^k}$.
\end{lemma}

\begin{proof}
Perturb $u_i^k$ to $v^k$, and the line segments move linearly from $\overline{u_i^{k-1}u_i^k}$ to  $\overline{u_i^{k-1}v^k}$, and from $\overline{u_i^ku_i^{k+1}}$ to $\overline{v^ku_i^{k+1}}$. Since $|u^k_i-v^k|<\frac{r}{2}$ and $\mu(\overline{u^{k-1}_iu^k_i}, \mathcal{C}^k)<\frac{r}{2}$ (Condition (3) in Lemma~\ref{lem:cds}), the perturbation stays inside $\Gamma^k_{\epsilon}$ which has a radius $r$. So during the perturbation, $\overline{u^{k-1}_i,u^k_i}$ and $\overline{u^k_i,u^{k+1}_i}$ do not intersect any line segments of $L_i$, possibly except their consecutive segments. But note that for each $k$, $u_i^k, v^k \in \Gamma_{\epsilon}^k \cap \Gamma_{\epsilon}^{k+1}$. An easy geometric analysis shows that this restricted area of the perturbation precludes the possibility for $\overline{u^{k-1}_i,u^k_i}$ and $\overline{u^k_i,u^{k+1}_i}$ intersecting their consecutive segments. So the perturbation does not cause intersections, and hence preserves the ambient isotopy. 
\end{proof}


\begin{theorem}[Isotopic Convergence Theorem]\label{thm:distcurv}
If  $\{C_i\}_{1}^{\infty}$ converges to $\mathcal{C}$ pointwise and in total curvature, then there exists an integer $N$ such that $C_i$ is ambient isotopic to $\mathcal{C}$ for all $i \geq N$.
\end{theorem}

\begin{proof}
For each $C_i$ and $\epsilon>0$, there exists an inscribed PL curve $L_i$ of $C_i$ such that $L_i$ is sufficiently close (bounded by $\epsilon$) to $C_i$ pointwise and in total curvature by Lemma~\ref{lem:app}, and ambient isotopic to $C_i$ by Theorem~\ref{thm:insc}. Since ambient isotopy is an equivalence relation \cite{Livingston1993}, we now rely on Theorem~\ref{thm:insc} to consider, without loss of generality, a sequence of PL curves $\{L_i\}_1^{\infty}$ instead of $\{C_i\}_{1}^{\infty}$.

Note that $T_{\kappa}(\mathcal{C}^k) < \frac{\pi}{2}$ and the polyline $\bigcup_{k=1}^n \overline{v^{k-1}v^k}$ lies inside of the tubular neighborhood (since $CH(\mathcal{C}^k) \subset \Gamma_{\mathcal{C}}$). By the proof of Theorem~\ref{thm:insc} we know that these are sufficient conditions for $\mathcal{C}$ being ambient isotopic to $\bigcup_{k=1}^n \overline{v^{k-1}v^k}$. By the equivalence relation of ambient isotopy, Lemma~\ref{lem:tpbk} implies that $\mathcal{C}$ is ambient isotopic to $\bigcup_{k=1}^n \overline{u_i^{k-1}u_i^k}$, and Lemma~\ref{lem:dtel} further implies that $\mathcal{C}$ is ambient isotopic to $L_i$. 
\end{proof}

\section{Some Conceptual Algorithms and Potential Applications}
\label{ssec:alg}
The Isotopic Convergence Theorem has both theoretical and practical applications. Theoretically, it formulates criteria to show the same knot type in knot theory. Practically, it provides rigorous theoretical foundations to extend current algorithms in computer graphics and visualization to much richer classes of curves than the splines already investigated \cite{TJP08}.

The following are the general procedures derived from our Theorem~\ref{thm:insc} and Theorem~\ref{thm:distcurv}. For a specific problem, further algorithmic development will depend upon characteristics of the class of curves. If the curve is ``nice" in the sense that the total curvature and the radius of a tubular surface is easy to compute, then it is easy to develop an algorithm. Such ``nice" curves include a rational cubic spline parameterized by arc length, for which the total curvatures can be easily computed, and the radius of a tubular surface can be found according to an existing result \cite{Maekawa_Patrikalakis_Sakkalis_Yu1998}. Otherwise, for some other curves, the computation of the total curvatures and the radius of a tubular neighborhood may be difficult and is beyond the scope of the details considered here, even while the theorems provide a broad framework within which these subtleties can be considered.

\subsection{Using PL knots to represent smooth knots}\label{ssec:pkrsk}

Based on Lemma~\ref{lem:app} and Theorem~\ref{thm:insc}, a procedure can be designed such that it takes a smooth knot as input and picks finitely many points on it to form an ambient isotopic PL knot. We call this a \textit{PL representation} of the smooth knot. 

Recall that two criteria are sufficient for the isotopy between a compact $C^2$ curve $\mathcal{C}$ and its inscribed PL curve $\mathcal{L}$:
\begin{enumerate}
\item Each sub-curve of $\mathcal{C}$ determined by two consecutive vertices of $\mathcal{L}$ has a  total curvature less than $\frac{\pi}{2}$.
\item The PL curve $\mathcal{L}$ lies inside of the tubular surface for $\mathcal{C}$ with radius $r$. (This can be achieved by making the Hausdorff distance between $\mathcal{L}$ and $\mathcal{C}$ less than $r$.) 
\end{enumerate}

\noindent {\it \large The Outline of Forming PL Representations:}
\begin{enumerate}
\item Select $\mathcal{C}(0)$ as the initial vertex of $\mathcal{L}$, denoted as $v_0$. 
\item Set\footnote{This $\epsilon$ value is not unique. Many others also work.} $\epsilon=0.1$. Select\footnote{It is not necessary for $T_{\kappa}(\mathcal{C}_{[0,t_1]})$ to be exactly $\frac{\pi}{2}-\epsilon$. For efficiency, it is fine to end up with a value not equal to $\frac{\pi}{2}-\epsilon$ as long as it is less than $\frac{\pi}{2}$. This aspect will require a subroutine to be developed that will likely vary over the class of curves considered and this detail is beyond the scope of the current investigation.} $t_1 \in [0,1]$ such that $T_{\kappa}(\mathcal{C}_{[0,t_1]})=\frac{\pi}{2}-\epsilon$. Let the second vertex of $\mathcal{L}$ be $\mathcal{C}(t_1)$, denoted as $v_1$. 
\item Similarly pick $t_2$ to obtain $v_2$. Continue until we reach the end point $\mathcal{C}(1)$, denoted as $v_n$. This process terminates because $\mathcal{C}$ is compact. In the end, we obtain an $\mathcal{L}$, and sub-curves of $\mathcal{C}$ with total curvatures being less than $\frac{\pi}{2}$.
\item Verify if the Hausdorff distance between $\mathcal{L}$ and $\mathcal{C}$ is less than $r$; If not, then select midpoints: $\mathcal{C}(\frac{t_j+t_{j+1}}{2})$ for $j\in \{0,1,\ldots,n-1\}$, denoted as $v_{\frac{2j+1}{2}}$, to form a new inscribed PL curve determined by vertices: $$\{v_0, v_{\frac{1}{2}}, v_1, v_{\frac{3}{2}},v_2, \ldots, v_{n-1}, v_{\frac{2n-1}{2}},v_n\}.$$ 
\item Repeat 4 until the Hausdorff distance between $\mathcal{L}$ and $\mathcal{C}$ is less than $r$. This process of selecting midpoints implies the pointwise convergence. So this process terminates.
\end{enumerate}

\subsection{Testing isotopic convergence}
For a $C^2$ curve $\mathcal{C}$ and a sequence $\{C_i\}_1^{\infty}$ of piecewise $C^2$ curves, where $\{C_i\}_1^{\infty}$ converges to $\mathcal{C}$ pointwise and in total curvature, we shall design a procedure to determine a positive integer $N$ such that whenever $i \geq N$, $C_i$ is ambient isotopic to $\mathcal{C}$. Use $T_{\kappa}(\cdot)$ to denote the total curvature, $CH(\cdot)$ the convex hull, and $\mu(\cdot)$ the Hausdorff distance. \\
\\
\noindent {\it \large The Outline of Testing Isotopic Convergence:}
\begin{enumerate}
\item Divide $\mathcal{C}$ into sub-curves $\mathcal{C}^k$ for $k=1,\cdots,n$ such that
\begin{itemize}
\item $T_{\kappa}(\mathcal{C}^k) < \frac{\pi}{2}$; and
\item $CH(\mathcal{C}^k) \subset \Gamma^k$.
\end{itemize}
\item  Set $i : =1$.
\item  Use the above technique to form a \textit{PL Representation} $L_i$ for the piecewise $C^2$ curve $C_i$ such that $L_i$ is ambient isotopic to $C_i$. 
\item  Let $L_i^k$ be the sub-curve of $L_i$ corresponding to $\mathcal{C}^k$. Denote the end points of $L_i^k$ as $u^{k-1}_i$ and $u^k_i$, and the corresponding end points of $\mathcal{C}^k$ as $v^{k-1}$ and $v^k$.Verify three criteria: 
\begin{itemize}
\item $T_{\kappa}(L_i^k) < \frac{\pi}{2}$;
\item $CH(L_i^k) \subset \Gamma^k_{\epsilon}$; and
\item $|u^k_i-v^k|<\frac{r}{2}$ and $\mu(\overline{u^{k-1}_iu^k_i}, \mathcal{C}^k)<\frac{r}{2}$. 
\end{itemize}
\item  If these criteria are satisfied, then let $N:=i$ and stop. Otherwise let $i:=i+1$ and go to (3).
\end{enumerate}
We know that $\{C_i\}_1^{\infty}$ pointwise converges to $\mathcal{C}$ and uniformly converges to $\mathcal{C}$ in total curvature (Lemma~\ref{lem:loc-cov}), so there exists a finite $i$ such that these three criteria are achieved, which means the above process terminates. By Isotopic Convergence Theorem, we obtain the ambient isotopy.

\subsection{A potential application for molecular simulations}\label{sec:avs}

It is often of interest to consider geometric models that are perturbed over time. For chemical simulations of macro-molecules, the algorithms in high performance computing (HPC) environments will produce voluminous numerical data describing how the molecule twists and writhes under local chemical and kinetic changes. These are reflected in changed co-ordinates of the geometric model, called perturbations. To produce a scientifically valid visualization, it is crucial that topological artifacts are not introduced by the visual approximations \cite{TJP08}. A primary distinction between the approximations created here in Section~\ref{ssec:pkrsk} and those based on Taylor's Theorem \cite{TJP08} is the expression here of the upper bound for  total curvature to be less than $\frac{\pi}{2} - \epsilon$.  

An earlier perturbation result is limited to PL curves  \cite{L.-E.Andersson2000}. Other ambient isotopic approximation methods rely upon the curve being a spline \cite{JL-bez-iso, Moore_Peters_Roulier2007}. The example in the next section needs not be a spline. Our approach provides a more general result of sufficient conditions for both approximation and perturbations which do not change topological features. Precisely, our Isotopic Convergence Theorem implies that as long as the convergence criterion of pointwise convergence and convergence in total curvature is satisfied, then ambient isotopy is preserved. 

\subsection{A representative example of offset curves}\label{sec:reoc}
Offset curves are defined as locus of the points which are at constant distant along the normal from the generator
curves \cite{Maekawa1999}. It is well-known \cite[p.~553 ]{Piegl} that offsets of spline curves need not be splines. They are widely used in various applications, and the related approximation problems were frequently studied. A literature survey on offset curves and surfaces prior to 1992 was conducted by Pham \cite{Pham1992}, and another such survey between 1992 and 1999 was given by Maekawa \cite{Maekawa1999}. Here we show a representative example as a catalyst to ambient isotopic approximations of offset curves. 

Let $\mathcal{C}(t)$ be a compact, regular, $C^2$, simple, space curve parametrized on $[a,b]$, whose curvature $\kappa$ never equals $1$. Then define an offset curve by
$$\Omega(t) = \mathcal{C}(t) + N(t),$$
where $N(t)$ is the normal vector at $t$, for $t \in [a,b]$. 

For example, let $\mathcal{C}(t)=(2\cos t, 2\sin t, t)$ for $t\in[0,2\pi]$ be a helix, then it is an easy exercise for the reader to verify that the above assumptions of $\mathcal{C}$ are satisfied, with $\kappa=\frac{2}{5}$. Furthermore, it is straightforward to obtain the offset curve $\Omega(t) = (\cos t, \sin t, t)$. 

We first show that $\Omega(t)$ is regular. Let $s(t)=\int_{a}^t |\mathcal{C}'(t)| dt$ be the arc-length of $\mathcal{C}$. Then by Frenet-Serret formulas \cite{DoCarmo1976} we have
$$\Omega'(t)=\mathcal{C}'(t)+N'(t)$$
$$=\frac{ds}{dt}T + (-\kappa T+\tau B)\frac{ds}{dt}=(1-\kappa)\frac{ds}{dt} T+\tau \frac{ds}{dt} B,$$
where $T$ and $B$ are the unit tangent vector and binormal vector respectively. 
Since $T \perp B$, if $(1-\kappa)\frac{ds}{dt} \neq 0$ then $\Omega'(t) \neq 0$. But $(1-\kappa)\frac{ds}{dt} \neq 0$ because $\kappa \neq 1$ and $\mathcal{C}(t)$ is regular by the assumption. Thus $\Omega(t)$ is regular. 

Now we define a sequence $\{ \Omega_i(t) \}_{i=1}^{\infty}$ to approximate $\Omega(t)$ by setting
$$\Omega_i (t) = \mathcal{C}(t) + \frac{i-1}{i}N(t).$$
It is obvious that $\{ \Omega_i(t) \}_{i=1}^{\infty}$ pointwise converges to $\Omega(t)$. For the convergence in total curvature, note that $\lim_{i \rightarrow \infty} \Omega'_i(t) = \Omega'(t)$, $\lim_{i \rightarrow \infty} \Omega''_i(t) = \Omega''(t)$, and $|\Omega'(t)| \neq 0$ due to the regularity of $\Omega(t)$. Therefore
$$\lim_{i \rightarrow \infty}  \frac{ \Omega'_i(t) \times \Omega''_i(t)}{|\Omega'_i(t)|^3}= \frac{ \Omega'(t) \times \Omega''(t)}{|\Omega'(t)|^3}.$$
The convergence in total curvature follows. 

Consequently, by\footnote{Since the example satisfies $C^0$ and $C^1$ convergence and the paper \cite{Sullivan2008} shows ambient isotopy under $C^0$ and $C^1$ convergence, the ambient isotopy for this example also follows from the previous result \cite{Sullivan2008}. Here our purpose is to use it as a representation to show how the Isotopic Convergence Theorem can be applied.} the Isotopic Convergence Theorem (Theorem~\ref{thm:distcurv}), we conclude that there exists a positive integer $N$ such that $\Omega_i(t)$ is ambient isotopic to $\Omega(t)$ whenever $i > N$. 

\section{Conclusion}
We derived the Isotopic Convergence Theorem by topological and geometric techniques, as motivated by  applications for knot theory, computer graphics, visualization and simulations. 

Future research may use the Isotopic Convergence Theorem in knot classification, since it provides a method to pick finitely many points from a given knot, where the set of finitely many points determines the same knot type. 

\section*{Acknowledgments}
The authors thank Professor Maria Gordina, Professor John M Sullivan, Dr. Chen-yun Lin and  the referee for stimulating questions and insightful comments about this manuscript.

The authors express their appreciation for partial funding from the National Science Foundation under grants CMMI 1053077 and CNS 0923158.  All expressions here are of the authors, not of the National Science Foundation.  The authors also express their appreciation for support from IBM under JSA W1056109, where statements in this paper are the responsibility of the authors, not of IBM.
\bibliographystyle{plain}
\bibliography{ji-tjp-biblio}

\end{document}